\documentclass[a4paper]{article}

\usepackage[utf8]{inputenc}
\usepackage{amsmath, amssymb, amsthm, xypic, mathtools}
\usepackage{comment}
\usepackage{xcolor}
\xyoption{rotate}

\topmargin = - 1.5 cm
\newtheorem{theorem}{Theorem}[section]
\newtheorem{proposition}[theorem]{Proposition}
\newtheorem{definition}[theorem]{Definition}
\newtheorem{lemma}[theorem]{Lemma}
\newtheorem{corollary}[theorem]{Corollary}

\theoremstyle{definition}

\newtheorem*{acknowledgement}{Acknowledgement}

\newcommand{\im}{\mathop{\mathrm{im}}}
\newcommand{\coker}{\mathop{\mathrm{coker}}}
\newcommand{\id}{\mathrm{id}}
\newcommand{\Sph}{\mathcal{S}}
\newcommand{\Aut}{\mathrm{Aut}}
\newcommand{\D}{\mathrm{D}}
\renewcommand{\O}{\mathrm{O}}
\newcommand{\Z}{\mathbb{Z}}
\newcommand{\Stab}{\mathrm{Stab}}
\newcommand{\Hom}{\mathrm{Hom}}
\renewcommand{\Dc}{\mathcal{D}}
\renewcommand{\P}{\mathcal{P}}

\newcommand{\NS}{\mathrm{NS}}
\newcommand{\Ox}{\mathcal{O}_X}

\begin{document}

\begin{center}
{\large \bfseries \MakeUppercase{Central derived autoequivalences of K3 surfaces}} \\[0.5cm]
\MakeUppercase{A. Savelyeva}
\end{center}

\begin{abstract}
 We apply the theory of Bridgeland's stability conditions to describe the center of the group $\Aut(\D^b(X))$ of bounded  derived autoequivalences of a complex projective K3 surface.
\end{abstract}

\section{Introduction}

Let $X$ be a complex projective K3 surface and consider the group of autoequivalences $\Aut(\D^b(X))$ of its bounded derived category of coherent sheaves. In this work we describe the center of $\Aut(\D^b(X))$. More precisely, we prove the following theorem, which can be seen as an analogue of the classical result that the center of a mapping class group of a surface is trivial for genus greater than $2$.

\begin{theorem} \label{main_intro}
The center $Z$ of $\Aut(\D^b(X))$ is isomorphic to $\Aut^t(X) \times \Z \cdot [1]$.
\end{theorem}

Here $\Z \cdot [1]$ is the subgroup of $\Aut(\D^b(X))$ generated by the shift functor and $\Aut^t(X)$ is the group of transcendental automorphisms of $X$, which by definition consists of all $f \in \Aut(X)$ such that $f^* \in \O(H^2(X, \Z))$ acts trivially on the algebraic part $\mathrm{NS}(X)$ of the cohomology lattice. It is known to be a finite cyclic group.

The question was previously addressed in \cite{center}, where Barbacovi and Kikuta proved the inclusion $Z \subset \Aut^t(X) \times \Z \cdot [1]$ and deduced the other inclusion from Bridgeland's conjecture (see \cite[Conj.~1.2]{Bridgeland_k3}). We manage to avoid Bridgeland's conjecture by using Huybrechts' approach to stability conditions via spherical objects developed in \cite{spherical}. Our resul can be seen as further evidence for Bridgeland's conjecture, which so far has only been proved for Picard rank $\leqslant 1$ (see \cite{BB}, \cite{few}).\\[-5pt]

The following is the principal technical result of this work.

\begin{proposition} \label{technical_intro}
Suppose an autoequivalence $f \in \Aut(\Dc)$ of a triangulated K3 category $\Dc$ preserves all stable spherical objects for a given stability condition $\sigma \in \Stab(\Dc)$. Then $f$ preserves all rigid objects in $\Dc$.
\end{proposition}

A triangulated category $\Dc$ is called a K3 category if a square of the shift functor is a Serre functor for $\Dc$. In particular if $X$ is a K3 surface, the derived category of coherent sheaves $\D^b(X)$ is a K3 category. An object $E$ of a K3 category $\Dc$ is called spherical if $\dim \Hom(E, E[i]) = 1$ for $i = 0,\,2$ and $\dim \Hom(E, E[i]) = 0$ otherwise. For the information on Bridgeland stability conditions see \cite{Bridgeland_stability}, \cite{Bridgeland_k3} and \cite{huybrechts_stability}.

We deduce from Proposition \ref{technical_intro} that an autoequivalence $\Phi$ acts trivially on the set $\Sph$ of all spherical objects in $\D^b(X)$ if and only if $\Phi \in \Aut^t(X)$.

\begin{corollary}\label{spherical_intro}
Let $\tau \colon \Aut(\D^b(X)) \to \Aut(\Sph)$ be the map induced by the action of $\Aut(\D^b(X))$ on $\Sph$. Then the kernel $\ker \tau \subset \Aut(\D^b(X))$ is equal to $\Aut^t(X)$.
\end{corollary}

This answers the question raised at the end of \cite[App.~A]{spherical}. 

Denote by $\Aut_0(\D^b(X)) \subset \Aut(\D^b(X))$ the group of autoequivalences that induce a trivial action on the cohomology $H^*(X, \Z)$. Proposition \ref{technical_intro} also allows us to obtain the following.

\begin{corollary}
Consider an autoequivalence $\Phi \in \Aut_0(\D^b(X))$. Then $\Phi$ acts on the space of all stability conditions without fixed points.
\end{corollary}

Moreover, we modify the approach of Barbacovi and Kikuta by providing an alternative, elementary proof of \cite[Thr.~1.1]{center} not requiring any dg-categories or twisted complexes as in \cite[Sect.~2~\&~6]{center}.\\[-5 pt]

%The structure of the text is as follows. Section~2 is devoted to the proof of a technical result that will be used further in the context of stability conditions. In Section~3 we discuss several known results about stability conditions and spherical objects. We then develop a filtration that allows us to complete the proof of Proposition~\ref{technical_intro} using the results of Section~2. In Section~4 we deduce Corollary~\ref{spherical_intro} from Proposition~\ref{technical_intro} and the results from~\cite[App.~A]{spherical}. We complete the proof of Theorem~\ref{main_intro} in Section~5. In Section~6 the alternative proof of the inequality~\mbox{\cite[Thr.~1.1]{center}} is presented. \\[-5 pt]

For the rest of the text let us fix the following notation. For two objects $A$, $B$ in a linear triangulated category $\Dc$ and for an integer $i$ we denote by $(A, B)^i$ the vector space $\Hom_{\Dc}(A, B[i])$. For a morphism $f \colon A \to B$ we denote by $C(f)$ the cone of $f$, i.e. the object obtained by completing $f \colon A \to B$ to a distinguished triangle.

\begin{acknowledgement}
I am grateful to my research advisor D.\,Huybrechts for introducing me to the topic of stability conditions, plenty of useful discussions and multiple comments on the text.
\end{acknowledgement}

\section{Isomorphic rigid cones}\label{technical_part}

This section is devoted to the proof of the following technical proposition.

\begin{proposition} \label{isomorphic_cones}
Assume $A$ and $E$ are two rigid objects in a triangulated category $\Dc$ such that~$(E, A)^i = 0$ for $i \leqslant 0$. Then for every two non-zero morphisms $f, \, g \colon A[-1] \to E$, if the cones $C(f)$ and $C(g)$ of $f$ and $g$ are rigid, they are isomorphic.
\end{proposition}

To prove the proposition, we show that the sets $U \subset (A, E)^1$ of morphisms with isomorphic rigid cones are open in Zariski topology (see Proposition \ref{open_neighborhood}), and thus conclude by taking their intersection.

We start by proving the following lemmas.

\begin{lemma} \label{dimension_inequality}
Under the assumptions of Proposition \ref{isomorphic_cones}, for the two maps $\alpha \colon (A, A)^0 \to (A, E)^1$ and $\beta \colon (E, E)^0 \to (A, E)^1$ given by taking composition with $f$, the following inequality holds 
$$\dim (\im \alpha) + \dim (\im \beta) > \dim (A, E)^1.$$
\end{lemma}

\begin{proof}
Consider the following diagram
$$\xymatrix@C=0.5cm{
(E, A)^{-1} \ar[r] \ar[d] & (E, E)^0 \ar[r] \ar[d]^\beta & (E, C(f))^0 \ar[r] \ar[d] & (E, A)^0 \\
(A, A)^0 \ar[r]^\alpha \ar[d] & (A, E)^1 \ar[r] \ar[d] & (A, C(f))^1 \ar[r] \ar[d] & (A, A)^1 \\
(C(f), A)^0 \ar[r] \ar[d] & (C(f), E)^1 \ar[r] \ar[d] & (C(f), C(f))^1 \\
(E, A)^0 & (E, E)^1
}$$
where the morphisms are given by taking compositions with shifts of morphisms in the distinguished triangle
\begin{equation} \label{main_triangle}
\xymatrix{A[-1] \ar[r]^(0.6)f & E \ar[r] & C(f) \ar[r] & A.}
\end{equation}

The diagram commutes due to associativity of composition. Moreover, its rows and columns are exact as they are obtained by applying the correct $\Hom$-functors to the distinguished triangle \eqref{main_triangle}.

Let us now consider the diagram above with the conditions of the proposition inserted.

$$\xymatrix@C=0.5cm{
0 \ar[r] \ar[d] & (E, E)^0 \ar[r]^(0.45)\simeq \ar[d]^\beta & (E, C(f))^0 \ar[rr] \ar@{->>}[d] && 0 \\
(A, A)^0 \ar[r]^\alpha \ar[d]^(0.45)[@!-90]\simeq & (A, E)^1 \ar@{->>}[r] \ar@{->>}[d] & (A, C(f))^1 \ar[rr] \ar[d] && 0 \\
(C(f), A)^0 \ar@{->>}[r] \ar[d] & (C(f), E)^1 \ar[r] \ar[d] & 0 \\
0 & 0
}$$

From commutativity of the diagram we deduce that $\dim (\im \alpha) \geqslant \dim(C(f), E)^1 = \dim (\coker \beta)$ and that $\dim (\im \beta) \geqslant \dim(A, C(F))^1 = \dim (\coker \alpha)$. Actually, the inequalities above are strict, as the images of $\alpha$ and $\beta$ intersect in $\alpha(\id) = \beta(\id) = f \ne 0$.

We thus obtain $$2(\dim (\im \alpha) + \dim (\im \beta)) > \dim (\coker \alpha) + \dim (\coker \beta) + \dim (\im \alpha) + \dim (\im \beta) = 2\dim(A, E)^1,$$
which completes the proof.
\end{proof}

The following lemma is a simple linear algebra fact that will help us to apply Lemma \ref{dimension_inequality}.

\begin{lemma} \label{technical_inequality}
Let $M_1 \colon V_1 \to W$ and $M_2 \colon V_2 \to W$ be linear maps of vector spaces and let $v_1 \in V_1$, $v_2 \in V_2$ be such that $M_1(v_1) = M_2(v_2)$. Suppose furthermore that $$\dim (\im M_1) + \dim (\im M_2) > \dim W.$$
Then for every Zariski open neighborhoods $U_1 \subset V_1$ and $U_2 \subset V_2$ of $v_1$ and $v_2$ there exists a Zariski open neighborhood $\mathcal{U} \subset \Hom(V_2, W)$ of $M_2$ such that for every $M_2' \in \mathcal{U}$ the images $M_1(U_1)$ and~$M_2'(U_2)$ intersect.
\end{lemma}

\begin{proof}
Note that all linear maps $M \colon V_2 \to W$ are parameterised by their graphs, which form a Zariski open subset in $\mathrm{Gr}(\dim V_2, \, \dim V_2 + \dim W)$.

Denote by $\Lambda$ the subset $(\im M_1) \times V_2$ of $W \times V_2$ and consider its Zariski open subset $$U = (M_1(U_1) \times V_2) \cap (W \times U_2).$$ Note that a map $M_2' \colon V_2 \to W$ satisfies $M_2'(U_2) \cap M_1(U_1) \ne 0$ if and only if its graph intersects $U \subset V_2 \times W$.

As $U$ is a Zariski open subset of $\Lambda$, the set of all elements of $\mathrm{Gr}(\dim V_2, \, \dim V_2 + \dim W)$ that intersect $U$ is a Zariski open subset of those that have a non-zero intersection with $\Lambda$. It remains to prove that every element of $\mathrm{Gr}(\dim V_2, \, \dim V_2 + \dim W)$ intersects with $\Lambda$, which follows from the inequality
$$\dim \Lambda = \dim (\im M_1) + \dim V_2 \geqslant \dim (\im M_1) + \dim (\im M_2) > \dim W.$$ \\[-1.2 cm]
\end{proof}

Now we are ready to prove the following.

\begin{proposition} \label{open_neighborhood}
Assume $A$ and $E$ are two rigid objects of a triangulated category $\Dc$ such that~$(E, A)^i = 0$ for $i \leqslant 0$ and $f \colon A[-1] \to E$ is a non-zero morphism with a rigid cone. Then there exists a Zariski open neighbourhood $U \subset (A, E)^1$ of $f$ such that for every $f' \in U$ there is an isomorphism $C(f) \cong C(f')$.
\end{proposition}

\begin{proof}
Note that the sets $U_E \subset (E, E)^0$ and $U_A \subset (A, A)^0$ consisting of isomorphisms are Zariski open. We apply Lemma \ref{technical_inequality} to $\alpha \colon (A, A)^0 \to (A, E)^1$ and $\beta \colon (E, E)^0 \to (A, E)^1$, the vectors $v_1$ and $v_2$ being $\id_A \in U_A$ and $\id_E \in U_E$ correspondingly. The inequality on the image dimensions follows from Lemma \ref{dimension_inequality}. 

Consider the Zariski open neighborhood $\mathcal{U} \subset \Hom((E, E)^0, (A, E)^1)$ of $\beta$ obtained by Lemma~\ref{technical_inequality}. Note that there exists a linear map $\Lambda \colon (A, E)^1 \to \Hom((E, E)^0, (A, E)^1)$ that takes $\psi \in (A, E)^1$ to~$\circ \, \psi \colon (E, E)^0 \to (A, E)^1$. In particular, $\Lambda(f) = \beta$. Denote by $U \subset (A, E)^1$ the preimage of $\mathcal{U}$ under $\Lambda$. As the preimage of a Zariski open set under a linear map it is Zariski open and it contains $f$. We claim that for every $f' \in U$ the cones of $f'$ and $f$ are isomorphic.

As $\Lambda(f') \in \mathcal{U}$, there exist $a \in U_A$ and $e \in U_E$ such that $\alpha(a) = \Lambda(f')(e)$, which means $f \circ a = e \circ f'$. Therefore, there exists a morphism of distinguished triangles
$$\xymatrix{
A[-1] \ar[r]^(0.55){f'} \ar[d]^a & E \ar[r] \ar[d]^e & C(f') \ar[r] \ar[d]^i & A \ar[d]^{a[1]} \\
A[-1] \ar[r]^(0.55)f & E \ar[r] & C(f) \ar[r] & A.
}$$

As $a$ and $e$ were constructed to be isomorphisms, so is $i$.
\end{proof}

We can now complete the proof of Proposition \ref{isomorphic_cones}.

\begin{proof}[Proof of Proposition \ref{isomorphic_cones}]
Let $f, \, g \in (A, E)^1$ be two morphisms with rigid cones. Consider subsets $U_f, \, U_g \subset (A, E)^1$ obtained by Proposition \ref{open_neighborhood} from $f$ and $g$. As they are both Zariski open, they have a non-empty intersection. Hence, for an element $h \in U_f \cap U_g$ Proposition \ref{open_neighborhood} gives 
$$C(f) \cong C(h) \cong C(g).$$ \\[-1.2 cm]
\end{proof}

\section{Stability conditions preserving rigid objects}

In this section we discuss the interactions of stability conditions and spherical and rigid objects in order to prove the following proposition.

\begin{proposition}\label{preserves_rigid_objects}
Suppose an autoequivalence $f \in \Aut(\Dc)$ of a triangulated K3 category $\Dc$ preserves all stable spherical objects for a given stability condition $\sigma \in \Stab(\Dc)$. Then $f$ preserves all rigid objects in $\Dc$.
\end{proposition}

For the proof of the following lemma see \cite[Lem. 2.7]{few}.

\begin{lemma}\label{ext^1_inequality}
For a triangulated category $\Dc$ consider a distinguished triangle 
$$\xymatrix{
A \ar[r] & E \ar[r] & B \ar[r] & A[1]
}$$
such that $(A, B)^r = (B, B)^s = 0$ for $r \leqslant 0$ and $s < 0$. Then $\dim(A, A)^1 + \dim(B, B)^1 \leqslant \dim(E, E)^1$.
\end{lemma}

\begin{corollary} \label{filtration_corollary}
Given a rigid object $E \in \Dc$ let
\begin{equation}\label{filtration_1}
\xymatrix{
& E_1 \ar[rr] \ar@{=}[ld] & & E_2 \ar[r] \ar[ld] & \ldots \ar[r]  & E_{n-1} \ar[rr] & & E_n \ar[ld] \ar@{=}[r] & E\\
A_1 & & A_2 \ar[ul]^{[1]} & & & & A_n \ar[ul]^{[1]}}
\end{equation}
be such that all $A_i$ are semistable and $(E_i, A_{i+1})^k = 0$ for every $k \leqslant 0$. Then the objects $E_i$ and $A_i$ are rigid for every $i$.
\end{corollary}

\begin{proof}
Note that for every semistable object $A_i$ of phase $\varphi_i$ and for every $s < 0$ the object $A_i[s]$ is semistable of phase $\varphi_i + s < \varphi_i$ and thus $(A_i, A_i)^s = 0$ for $s < 0$. 

We can now apply Lemma \ref{ext^1_inequality} to the distinguished triangles of the filtration \eqref{filtration_1} to obtain 
\begin{equation} \label{local_inequality}
\dim(E_{i-1}, E_{i-1})^1 + \dim(A_i, A_i)^1 \leqslant \dim(E_i, E_i)^1
\end{equation}
for all $i \geqslant 2$. This yields $\sum_i \dim(A_i, A_i)^1 \leqslant \dim(E, E)^1 = 0$, so $(A_i, A_i)^1 = 0$ for all $i$. Now, using \eqref{local_inequality}, we can prove $(E_i, E_i)^1 = 0$ by descending induction from $i = n$.
\end{proof}

A special case of a filtration in Corollary \ref{filtration_corollary} is the Harder--Narasimhan filtration, though we will also apply it to a little bit thinner variant of it.

The following proposition is deduced from Corollary \ref{filtration_corollary} applied to the Harder--Narasimhan filtration in \cite[Prop. 2.9]{few}.

\begin{proposition} \label{stable_factors}
All stable factors of a rigid object $E \in \Dc$ are spherical.
\end{proposition}

The following lemma will be used in the proof of Proposition \ref{preserves_rigid_objects} to construct a filtration that on the one hand satisfies rigidity properties as the Harder--Narasimhan filtration, but on the other hand has simple enough adjoint factors.

\begin{lemma}\label{good_decomposition}
For any rigid object $E$ there exists a sequence of real numbers $\varphi_1 \geqslant \varphi_2 \geqslant \ldots \geqslant \varphi_n$ and a filtration 
$$\xymatrix@C=0.7cm{
& E_1 \ar[rr] \ar@{=}[ld] & & E_2 \ar[r] \ar[ld] & \ldots \ar[r]  & E_{n-1} \ar[rr] & & E_n \ar[ld] \ar@{=}[r] & E\\
A_1^{\oplus m_1} & & A_2^{\oplus m_2} \ar[ul]^{[1]} & & & & A_n^{\oplus m_n} \ar[ul]^{[1]}
}$$
with $A_i$ being stable spherical objects of phases $\varphi_i$ such that $(E_i, A_{i+1})^0 = 0$ for every $i$.
\end{lemma}

\begin{proof}
We proceed by the induction on the length of the Jordan--H\"older filtration. 

Let $A$ be the last factor of the Jordan--H\"older filtration of $E$. It is spherical by Proposition \ref{stable_factors}. Let $\varphi$ be its phase. Denote by $n$ the dimension $\dim (E, A)^0$ and let $\alpha \colon E \to A^{\oplus n} = (E, A)^0 \otimes A$ be the evaluation morphism.

Complete $\alpha$ to a distinguished triangle. This gives
$$\xymatrix{A^{\oplus n}[-1] \ar[r]^(0.6)\gamma & E' \ar[r]^\beta & E \ar[r]^(0.4)\alpha & A^{\oplus n}
}$$

Let us prove that there are no morphisms from $E'$ to $A$. By construction, $\circ \, \alpha$ induces an injection $(A^{\oplus n}, A)^0 \to (E, A)^0$, therefore $\circ \, \gamma$ induces a surjection $(A^{\oplus n}, A)^1 \to (E', A)^0$. However, $(A, A)^1 = 0$, hence $(A^{\oplus n}, A)^1 = 0$, and therefore $(E', A)^0 = 0$.

Note that, moreover, there are no morphisms from $E'$ to the objects of smaller phase than $A$. Hence, the maximal phase of $E'$ is at least $\varphi$. So the Jordan--H\"older filtration of $E'$ can be considered a part of Jordan--H\"older filtration of $E$. In particular it is strictly shorter, so we can apply the induction statement.
\end{proof}

Now we are ready to prove Proposition \ref{preserves_rigid_objects}.

\begin{proof}[Proof of Proposition \ref{preserves_rigid_objects}]
We proceed by induction on the length $n$ of the decomposition obtained in Lemma \ref{good_decomposition}. When $n = 1$ the object $E$ is isomorphic to a direct sum of stable spherical objects, so it is preserved by $f$.

Now consider a rigid object $E$ with decomposition 
$$\xymatrix@C=0.7cm{
& E_1 \ar[rr] \ar@{=}[ld] & & E_2 \ar[r] \ar[ld] & \ldots \ar[r]  & E_{n-1} \ar[rr] & & E_n \ar[ld] \ar@{=}[r] & E\\
A_1^{\oplus m_1} & & A_2^{\oplus m_2} \ar[ul]^{[1]} & & & & A_n^{\oplus m_n} \ar[ul]^{[1]}
}$$
of length $n$. %By Lemmas \ref{of_same_phase} and \ref{no_morphisms_phases} 
We can apply Corollary \ref{filtration_corollary} to deduce $(E_{n-1}, E_{n-1})^1 = 0$ and then by induction obtain $f(E_{n-1}) = E_{n-1}$.

Let $\psi \colon A_n^{\oplus m_n}[-1] \to E_{n-1}$ be the morphism such that 
$E_n = C(\psi)$. Suppose $\psi \ne 0$. In this case, as $A_n$ and $E_{n-1}$ are preserved by $f$ and $f(E_n)$ is rigid, the desired isomorphism follows from Proposition \ref{isomorphic_cones}.

In the case when $\psi = 0$, as $f(\psi) = 0$, we deduce $E_n = C(f(\psi)) = f(C(\psi)) = f(E_n)$. 
\end{proof}

\section{Stability conditions on K3 surfaces}

Let $X$ be a projective K3 surface. As before denote by $\Sph$ the set of all spherical objects of $\D^b(X)$. In this section we use the technical result obtained in Section \ref{technical_part} to describe the kernel of the morphism $\tau \colon \Aut(\D^b(X)) \to \Aut(\Sph)$ induced by the action of autoequivalences on the set of spherical objects in $\D^b(X)$.

Denote by $\Sigma$ the distinguished connected component of the space of stability conditions as introduced by Bridgeland in \cite{Bridgeland_k3} and by $\Aut^\Sigma(\D^b(X))$ the subgroup of $\Aut(\D^b(X))$ that preserves~$\Sigma$. Let $\kappa \colon \Aut^\Sigma(\D^b(X)) \to \Aut(\Sigma)$ be the corresponding action.

For the proof of the following lemma see \cite[Lem.~A.3]{spherical}.

\begin{lemma} \label{ker_kappa}
The kernel of $\kappa$ coincides with the group of transcendental automorphisms of $X$: $$\ker(\kappa \colon \Aut^\Sigma(\D^b(X)) \to \Aut(\Sigma)) = \Aut^t(X).$$
\end{lemma}

The following proposition is a part of an approach to stability conditions via spherical objects developed by Huybrechts in \cite{spherical}. For the proof see \cite[Thr. 3.1]{spherical}. Note, however, that the conditions of the proposition differ slightly from those in \cite{spherical}. In particular we ask $\sigma$ to be geometric but do not need $\sigma'$ to lie in the distinguished connected component $\Sigma$. For the proof of the statement in this form see \cite[Sect. 3.1]{spherical}.
Semistability in the statement can be replaced by stability due to \cite[Prop. 2.5]{spherical}. 

For the definition of geometric stability conditions see \cite{Bridgeland_k3}.

\begin{proposition} \label{spherical_stability}
Let $\sigma =(Z, \P) \in \Sigma$ be a geometric stability condition. Then for an arbitrary stability condition $\sigma' = (Z', \P') \in \Stab(X)$ the following are equivalent
\begin{itemize} \setlength\itemsep{-2pt}
\item[\rm{(i)}] $\sigma' = \sigma$;
\item[\rm{(ii)}] $Z = Z'$ and every spherical $A \in \D^b(X)$ is $\sigma$-stable if and only if it is $\sigma'$-stable.
\end{itemize}
\end{proposition}
\

The proof of the following lemma is an adoption of a classical argument of Mukai to the case of Bridgeland stability conditions and can be found in \cite{spherical}.

\begin{lemma}\label{preserving_spherical}
Let $f \in \Aut(\D^b(X))$ be an autoequivalence acting trivially on the algebraic part $\NS(X)$ of Hodge lattice, and let $\sigma \in \Stab(X)$ be a stability condition such that $f(\sigma) = \sigma$. Then every $\sigma$-stable spherical object is preserved by $f$.
\end{lemma}

\begin{proof}
Let $A$ be a $\sigma$-stable spherical object of phase $\varphi$. As $f$ preserves $\sigma$, the spherical object~$f(A)$ is also $\sigma$-stable of the same phase, hence $(A, f(A))^i = 0$ for $i \leqslant -1$ and $i \geqslant 3$ due to Serre duality. Furthermore, as $f$ acts by identity on $\NS(X)$, the objects $A$ and $f(A)$ have the same Mukai vectors, therefore $\chi(A, f(A)) = 2$, thus there exists a non-zero morphism $\alpha \colon A \to f(A)$. As $A$ and $f(A)$ are stable objects of the same phase, it is an isomorphism.
\end{proof}

The following is now a direct consequence of Proposition \ref{preserves_rigid_objects}.

\begin{proposition} \label{acting_trivially_on_S}
An autoequivalence $f \in \Aut(\D^b(X))$ acts trivially on the set $\Sph$ of spherical objects if and only if it preserves pointwise the distinguished connected component $\Sigma$ of $\Stab(X)$.
\end{proposition}

\begin{proof}
Assume first that $f$ acts trivially on $\Sph$. All line bundles are spherical, hence $f$ pointwise preserves $\NS(X)$. In particular it preserves a stability function of every stability condition $\sigma \in \Sigma$. Proposition \ref{spherical_stability} applied to a geometric stability condition $\sigma \in \Sigma$ and to its image $f(\sigma)$ then yields $\sigma = f(\sigma)$. As $f$ preserves $\NS(X)$, it acts on $\Sigma$ as an automorphism of the covering (see \cite{Bridgeland_k3}) that has a fixed point, therefore it preserves $\Sigma$ pointwise. 

%It then follows from \cite[Lem. A.2]{spherical} that \mbox{$f \in \Aut^t(X)$.} Combining with Lemma \ref{ker_kappa} we obtain one inclusion.

For the converse, assume that $f$ preserves $\Sigma$ pointwise. Then $f$ acts as identity on $\NS(X)$, and we can apply Lemma \ref{preserving_spherical} to deduce that $f$ preserves all $\sigma$-stable spherical objects for any stability condition $\sigma \in \Sigma$. We conclude by Proposition \ref{preserves_rigid_objects}.
\end{proof}

Combining Proposition \ref{acting_trivially_on_S} and Lemma \ref{ker_kappa}, we obtain the following.

\begin{corollary} \label{main_cor}
Only transcendental autoequivalences act trivially on the set of spherical objects: $$\ker (\tau \colon \Aut(\D^b(X)) \to \Aut(\Sph)) = \Aut^t(X).$$
\end{corollary}

The following is also a corollary of the discussion above.

\begin{proposition}
Consider an autoequivalence $\Phi \in \Aut_0(\D^b(X))$. Then $\Phi$ acts on the space of all stability conditions without fixed points.
\end{proposition}

\begin{proof}
Suppose there exists a stability condition $\sigma \in \Stab(X)$ such that $\Phi(\sigma) = \sigma$. Lemma \ref{preserving_spherical} together with Proposition \ref{preserves_rigid_objects} then yield that $\Phi$ acts trivially on $\mathcal{S}$, and thus it pointwise preserves~$\Sigma$ by Proposition~\ref{acting_trivially_on_S}. Hence, due to Bridgeland's theorem \cite[Thr.~1.1]{Bridgeland_k3}, it is trivial.
\end{proof}

%extended corollary
\begin{comment}
\begin{corollary}
The morphism $f \in \Aut(\D^b(X))$ the following are equivalent
\begin{enumerate}
\setlength\itemsep{-2pt}
\item $f$ acts trivially on the set $\Sph$ of all spherical objects;
\item The action of $f$ on the space $\Stab(X)$ has an irrational fixed point $\sigma$;
\item There exists a full connected component $\Sigma'$ of $\Stab(X)$ preserved pointwise by $f$;
\item The distinguished connected component $\Sigma$ of $\Stab(X)$ is preserved pointwise by $f$;
\item $f \in \Aut^t(X)$.
\end{enumerate}
\end{corollary}
\end{comment}

\section{Describing the center}

This section is devoted to the description of the center of the group of derived autoequivalences of a projective K3 surface.

The following lemma is deduced in \cite[Thr. 4.6]{center} from the inequality discussed in Section \ref{inequality_section}, see Proposition \ref{inequality_proposition} or \cite[Thr. 4.3]{center}.

\begin{lemma} \label{same_twists}
Two spherical twists $T_{E_1}$ and $T_{E_2}$ corresponding to spherical objects $E_1$, $E_2$ of $\D^b(X)$ are equal if and only if $E_1 = E_2[k]$ for some integer $k$.
\end{lemma}

The following proposition can be deduced from the arguments by Barbacovi and Kikuta presented in \cite{center}, see the proof of part (i) of \cite[Thr. 8.1]{center}. For convenience of the reader we provide them here.

\begin{proposition} \label{one_inclusion}
The center $Z$ of $\Aut(\D^b(X))$ is contained in $\Aut^t(X) \times \Z \cdot [1]$, where $\Z \cdot [1]$ is the subgroup of $\Aut(\D^b(X))$ generated by the shift functor.
\end{proposition}

\begin{proof}
Consider an autoequivalence $\Phi$ in the center. It commutes with the spherical twist $T_{\Ox}$ around the trivial bundle and thus $T_{\Phi(\Ox)} = \Phi \circ T_{\Ox} \circ \Phi^{-1} = T_{\Ox}$. This yields $\Phi(\Ox) = \Ox[n]$ for some integer $n$ due to Lemma \ref{same_twists}, so the autoequivalence $\Phi \circ [-n]$ preserves $\Ox$. Moreover, as $\Phi \circ [-n]$ commutes with $\otimes \, \mathcal{L}$ for any line bundle $\mathcal{L}$, it leaves invariant all line bundles, and, in particular, acts on the graded ring $\bigoplus H^0(X, L^i)$ for an ample line bundle $L$. 

Consider an automorphism $f \in \Aut(X, L)$ corresponding to that action. Note that $\Phi\circ[-n]$ and $f$ define autoequivalences which are isomorphic on the full subcategory given by the ample sequence $\{L^i\}$, hence a result of Bondal and Orlov (see \cite{BO}) yields $\Phi\circ [-n] \cong f$. Moreover, as~$\Phi \circ [-n]$ preserves all line bundles, the morphism $f$ lies in $\Aut^t(X)$.
\end{proof}

It remains to prove the other inclusion, which is our main result.

\begin{theorem} \label{main}
The center $Z$ of $\Aut(\D^b(X))$ is isomorphic to $\Aut^t(X) \times \Z \cdot [1]$, where $\Z \cdot [1]$ is the subgroup of $\Aut(\D^b(X))$ generated by the shift functor.
\end{theorem}

\begin{proof}
By Proposition \ref{one_inclusion} it remains to prove that every $f \in \Aut^t(X)$ lies in the center. For an autoequivalence $\Phi \in \Aut(\D^b(X))$ denote by $\Phi_f$ the commutator $\Phi \circ f \circ \Phi^{-1} \circ f^{-1}$. Note that $\Phi_f$ acts trivially on the algebraic part of $H^*(X, \Z)$ as $f$ does. As the group of Hodge isometries of the transcendental part $T(X) \subset H^*(X, \Z)$ is abelian, the induced action of $\Phi_f$ is also trivial, so $\Phi_f \in \Aut_0(\D^b(X))$.

Moreover, as $f$ preserves all the spherical objects due to Corollary \ref{main_cor}, so $\Phi_f$ does. Therefore, by Proposition \ref{acting_trivially_on_S}, $f$ preserves pointwise the distinguished connected component $\Sigma$ of the space of stability conditions. Hence, due to Bridgeland's theorem \cite[Thr. 1.1]{Bridgeland_k3}, $\Phi_f = \id$, which concludes the proof.
\end{proof}

%\section{Further discussion}

\section{Note on an inequality of spherical twists}\label{inequality_section}

In this section we provide an alternative proof of \cite[Thr. 4.3]{center} which is elementary and does not use the dg-category tools developed for it in \cite{center}. 

As explained in \cite{center}, the considered inequality mimics the analogous inequality for Dehn twists of simple closed curves in mapping class groups of closed surfaces (see, for example, \cite[Prop.~3.4]{mcg}) and this is where the terminology comes from.

\begin{definition}
Given two objects $M$ and $N$ in a triangulated category $\Dc$ of finite type we define their intersection index $i(M,N)$ as 
$$i(M,N) \coloneq \sum_{j \in \Z}\dim(M, N)^j.$$
\end{definition}

Note that in the case when $\Dc$ is a K3 category, Serre duality yields $i(A, B) = i(B, A)$ for every $A$, $B \in \Dc$.

The following simple inequality on the intersection indexes follows directly from the long exact sequence obtained from applying $\Hom(E, -)$ to a distinguished triangle \eqref{another_triangle}.

\begin{lemma} \label{triang_ineq}
Given a distinguished triangle 
\begin{equation} \label{another_triangle}
\xymatrix{A \ar[r] & B \ar[r] & C \ar[r] & A[1],}
\end{equation}
for every object $E \in \Dc$ we have $i(E, B) \geqslant i(E, A) + i(E, C)$.
\end{lemma}

Before proving the inequality let us prove the following proposition.

\begin{proposition} \label{preliminary}
Given a spherical object $E$ in a triangulated category $\Dc$, for every $M \in \Dc$ there exists a sequence $C_i$ for $i \geqslant 1$ of objects such that \\[-18pt]
\begin{enumerate}
\setlength\itemsep{-2pt}
\item[\rm{(i)}] Every $C_i$ is an object of the triangulated subcategory $\langle E \rangle \subset \Dc$ generated by $E$;
\item[\rm{(ii)}] For every $i \geqslant 1$ there exists a distinguished triangle 
\begin{equation} \label{triangle}
\xymatrix{C_i \ar^-{\alpha_i}[r] & M \ar^-{\beta_i}[r] & T_E^i(M) \ar[r] & C_i[1],}
\end{equation}
such that the maps $\overline \beta_i^j \colon (E, M)^j \to (E, T_E^i(M))^j$, induced by $\beta_i$, are zero for every~$j$.
\end{enumerate}
\end{proposition}

\begin{proof}
Put $$C_1 = \bigoplus_j (E, M)^j \otimes E[-i].$$ The distinguished triangle \eqref{triangle} for $i = 1$ is then obtained from the definition of a spherical twist. Since the map $\alpha_1 \colon C_1 \to M$ is the evaluation map, every morphism $\varphi \colon E[-j] \to M$ can be lifted to a morphism $\overline \varphi \colon C_1 \to M$, hence $\overline \beta_1^j = 0$ for every $j$.

We proceed inductively as follows.
Consider the distinguished triangle
$$\xymatrix{C_n \ar^-{\alpha_n}[r] & M \ar^-{\beta_n}[r] & T_E^n(M) \ar^-{\gamma_n}[r] & C_n[1]}$$
provided by induction statement, and the distinguished triangle
$${\xymatrixcolsep{0.5cm}\xymatrix{\bigoplus_i (E, T_E^n(M))^i \otimes E[-i] \ar^-{ev}[r] & T_E^n(M) \ar^-f[r] & T_E^{n+1}(M) \ar[r] & \bigoplus_i (E, T_E^n(M))^i \otimes E[-i + 1]}}$$
obtained from the definition of a spherical twist. Consider the map
$$\gamma_n \circ \, ev \colon \bigoplus_i (E, T_E^n(M))^i \otimes E[-i] \to C_n[1]$$
and denote by $C_{n+1}[1]$ its cone. The induction statement implies that $C_n$ is an object of $\langle E \rangle$, hence so is $C_{n+1}$. 
We thus obtain a commutative diagram with three distinguished triangles
$$\xymatrix@!C = 1.5cm{
& & \bigoplus_i (E, T_E^n(M))^i \otimes E[-i] \ar_(0.45){ev}[d] \ar^(0.55){\gamma_n \circ \,ev}[dr] \\
C_n \ar^-{\alpha_n}[r] & M \ar^-{\beta_n}[r] & T_E^n(M) \ar_-f[d] \ar^{\gamma_n}[r] & C_n[1] \ar[dr] \\
& & T_E^{n+1}(M) & & C_{n+1}[1].
}$$
Due to octahedral axiom there exists a distinguished triangle
$$\xymatrix{C_{n+1} \ar[r] & M \ar^-{f \circ \beta_n}[r] & T_E^{n+1}(M) \ar[r] & C_{n+1}[1].}$$
Note that $\overline{(f \circ \beta_n)}^j \colon (E, M)^j \to (E, T_E^{n+1}(M))^j$ is zero for every $j$ as $\overline \beta_n^j = 0$ by induction.
\end{proof}

We can now proceed with the proof of the inequality. 

\begin{proposition} \label{inequality_proposition}
Let $\Dc$ be a triangulated K3 category. For every two objects $M$, $N \in \Dc$ and a spherical object $E \in \Dc$ the following inequality holds
$$i(M,E)i(N,E) \leqslant i(T_E^k(M),N) + i(M,N)$$
for every positive integer $k$.
\end{proposition}

\begin{proof}
Consider the distinguished triangle 
\begin{equation}\label{anothertriangle}
\xymatrix{C_k \ar^-{\alpha_k}[r] & M \ar^-{\beta_k}[r] & T_E^k(M) \ar[r] & C_i[1]}
\end{equation}
given by Proposition \ref{preliminary}. Lemma \ref{triang_ineq} then yields $i(N,M) + i(N,T_E^k(M)) \geqslant i(N, C_k)$. Moreover since $\bar \beta_k^j \colon (E, M)^j \to (E, T_E^i(M))^j$ is zero for every $j$, the long exact sequence obtained by applying $\Hom(E,-)$ to the distinguished triangle \eqref{anothertriangle} splits, hence for $N = E$ the inequality above is actually an equality $$i(C_k, E) = i(M, E) + i(T_E^k(M), E) = 2i(M, E).$$

Proposition \ref{preliminary} implies that $C_k$ is an object of a triangulated subcategory generated by $E$, hence $T_E(C_k) = C_k[-1]$. Then there exists a distinguished triangle 
$$\xymatrix{\bigoplus_i (E, C_k)^i \otimes E[-i] \ar^(0.73){ev}[r] & C_k \ar[r] & C_k[-1] \ar[r] & \bigoplus_i (E, C_k)^i \otimes E[-i + 1].}$$
Due to Lemma \ref{triang_ineq} we then obtain $$2i(C_k, N) \geqslant i(\bigoplus_i (E, C_k)^i \otimes E[-i], N) = i(E, C_k)i(E,N) = 2i(E, M)i(E,N)$$
which, together with the inequality above, yields
$$i(N,M) + i(N,T_E^k(M)) \geqslant i(N, C_k) \geqslant i(E, M)i(E,N).$$ \\[-1.2cm]
\end{proof}

{\scshape\'Ecole Normale Sup\'erieure, 45 rue d’Ulm, 75005 Paris, France}

\textit{Email address:} AnnSavelyeva57@gmail.com

\end{document}